\documentclass{article}

\usepackage{amsrefs,amssymb,amsmath,amsthm,amsfonts,epsfig,graphicx}
\usepackage[utf8]{inputenc}
\usepackage[T1]{fontenc}
\usepackage[english]{babel}
\usepackage{hyperref}
\hypersetup{
    colorlinks,
    citecolor=black,
    filecolor=black,
    linkcolor=black,
    urlcolor=black
}

\theoremstyle{plain}
\newtheorem{teo}{Theorem}[section]
\newtheorem{thm}[teo]{Theorem}
\newtheorem*{thm*}{Theorem}
\newtheorem{cor}[teo]{Corollary}
\newtheorem{lem}[teo]{Lemma}

\theoremstyle{definition}

\newtheorem{exa}[teo]{Example}
\newtheorem{rmk}[teo]{Remark}

\DeclareMathOperator{\diam}{diam}

\DeclareMathOperator{\dist}{dist}

\DeclareMathOperator{\card}{card}
\DeclareMathOperator{\homeos}{{\mathcal H}}

\DeclareMathOperator{\clos}{clos}

\newcommand{\azul}[1]{\textcolor{black}{#1}}

\newcommand{\expc}{\eta}

\newcommand{\R}    {\mathbb R}

\newcommand{\Z}  {\mathbb Z}

\newcommand{\gen}[1]{{\left<#1\right>}}

\renewcommand{\epsilon}{\varepsilon}

\author{Alfonso Artigue}
\title{Separating Homeomorphisms}
\date{\today}
\begin{document}

\maketitle

\begin{abstract}
We show that on a totally disconnected compact metric space every separating homeomorphisms is expansive except at periodic points.
We conclude that minimal separating homeomorphisms are expansive
and that every separating homeomorphism has asymptotic points.
We show that the only spaces admitting separating (or finite expansive) and recurrent homeomorphisms are finite sets. 
We apply our results to give a characterization of expansivity in terms 
of the expansivity of the cyclic group of powers of the homeomorphism.
\end{abstract}

\setcounter{tocdepth}{3}

\section{Introduction}

On a compact metric space $(X,\dist)$ consider a homeomorphism $f\colon X\to X$.
We say that $f$ is \emph{separating} if there is $\expc>0$ such that 
$\dist(f^n(x),f^n(y))\leq\expc$ for all $n\in\Z$ implies that $y=f^k(x)$ for some integer $k$.
This kind of dynamics was considered in \cites{Wine82,Wine85} 
for the study of the extensions of expansive homeomorphisms (see \S \ref{secSepHomeo} for the definition of expansive homeomorphism and some variations).
Separating homeomorphisms are also related to expansive flows in the following way. 
Let $\phi\colon\R\times X\to X$ be a flow (a continuous action of $\R$). 
A flow is \emph{kinematic expansive} \cite{ArKinExp} if for all $\expc>0$ there is $\delta>0$ such that 
if $\dist(\phi_t(x),\phi_t(y))\leq\expc$ for all $t\in\R$ then 
$y=\phi_s(x)$ for some $s\in (-\delta,\delta)$. 
In particular, if
$\dist(\phi_t(x),\phi_t(y))\leq\expc$ for all $t\in\R$ then 
$x$ and $y$ are in the same orbit, as for separating homeomorphisms (changing $\R$ by $\Z$).

In this paper we establish some links between separating and expansive homeomorphisms.
Our main result is Theorem \ref{teoSepFin}
where we show that if $f$ is separating and $X$ is totally disconnected 
then there is $\epsilon>0$ such that if $\dist(f^n(x),f^n(y))\leq\epsilon$ for all $n\in\Z$ then 
$x$ and $y$ are in a common periodic orbit.
In Example \ref{sepnoexp} we give a separating homeomorphism on a countable compact metric space that is not $N$-expansive. 
The example illustrates Theorem \ref{teoSepFin} 
and shows that the powers of a separating homeomorphism may not be separating,  
see Remark \ref{rmkPowers}.
In Corollaries \ref{corMin} and \ref{corSepAsymp} we conclude that minimal separating homeomorphisms are expansive
and that separating homeomorphisms have asymptotic points.
In Theorem \ref{teosepRecFin} we show that if $f$ is separating (or finite expansive) and recurrent then $X$ is a finite set. 
In Corollary \ref{corCyc}
we apply our results to give a characterization of expansivity in terms 
of the expansivity of the cyclic group of powers of the homeomorphism.

\section{Separating homeomorphisms}
\label{secSepHomeo}
Let $(X,\dist)$ be a compact metric space. 
A homeomorphism $f\colon X\to X$ is \emph{expansive} 
if there is $\expc>0$ such that if $x\neq y$ then $\dist(f^n(x),f^n(y))>\expc$ 
for some $n\in \Z$. In this case $\expc$ is an \emph{expansivity constant} for $f$. 
The \emph{orbit} of $x\in X$ is the set $$O(x)=\{f^n(x):n\in\Z\}.$$ 
As we said, $f$ is \emph{separating} \cites{Wine82,Wine85} if there is $\expc>0$ such that 
if $y\notin O(x)$ then $\dist(f^n(x),f^n(y))>\expc$ for some $n\in\Z$. 
In this case $\expc$ is a \emph{separating constant}.
The following result was proved by Wine and gives a fundamental link between expansive 
and separating homeomorphisms.

\begin{teo}[\cite{Wine82}]
 A separating homeomorphism $f$ is expansive if and only if 
 there is $\expc>0$ such that 
  if $x\in X$, $n\in \Z$ and $f^n(x)\neq x$ then there is 
  $r\in\Z$ such that $\dist(f^{r+n}(x),f^r(x))>\expc$.
\end{teo}

Let us introduce some forms of expansivity from the references \cites{CarCor,Mo12,MoSi,ArCa,Ka93}
that will be used throughout the paper.
Given $x\in X$ and $\expc>0$ define 
\[
 \Gamma_\expc(x)=\{y\in X:\dist(f^k(x),f^k(y))\leq\epsilon\text{ for all }k\in\Z\}.
\]
We say that the homeomorphism $f$ is: 
\begin{itemize}
\item $N$-\emph{expansive}       if $\exists$ $\expc>0$ s.t. $\card(\Gamma_\expc(x))\leq N,\forall x\in X$, 
\item \emph{finite expansive}    if $\exists$ $\expc>0$ s.t. $\card(\Gamma_\expc(x))< \infty, \forall x\in X$,
\item \emph{countably expansive} if $\exists$ $\expc>0$ s.t. $\card(\Gamma_\expc(x))\leq\card(\Z), \forall x\in X$,
\item \emph{cw-expansive}        if $\exists$ $\expc>0$ s.t. $\Gamma_\expc(x)$ is totally disconnected $\forall x\in X$.
\end{itemize}
In each case, we say that $\expc$ is a constant of the corresponding form of expansivity.
Table \ref{tablaExp} summarizes the variations of expansivity that we are considering.
\begin{table}[h]
\[
\begin{array}{ccccccc}
\text{exp} & \rightarrow & \text{separating} \\
\downarrow &&& \searrow\\
N\text{-exp} & \rightarrow & \text{finite exp} 
& \rightarrow &  \text{countably exp}
& \rightarrow &  \text{cw-exp}
\end{array}
\]
\caption{Hierarchy of expansivity on arbitrary metric spaces.}
\label{tablaExp}
\end{table}

The implications indicated by the arrows in Table \ref{tablaExp} are direct from the definitions. 
For example, every separating homeomorphism is countably expansive because, as $\Z$ is countable, 
every orbit is countable. 
Also, countable expansivity implies cw-expansivity because every non-trivial connected set is uncountable.

\begin{lem}
 \label{lemFinExpGamma}
 If $f$ is a homeomorphism of a compact metric space then:
 \begin{enumerate}
  \item $f$ is finite expansive if and only if there is $\epsilon>0$ such that every $x\in X$ is an isolated point of $\Gamma_\epsilon(x)$,
  \item $f$ is separating if and only if there is $\epsilon>0$ such that $\Gamma_\epsilon(x)\subset O(x)$ for all $x\in X$.
 \end{enumerate}
\end{lem}

\begin{proof}
If $\expc$ is a constant of finite-expansivity then $\Gamma_{\expc}(x)$ is finite for every $x\in X$ 
and $x$ is isolated in $\Gamma_{\expc}(x)$. 
Conversely, suppose that each $x\in X$ is isolated in $\Gamma_\epsilon(x)$.
We will show that $\epsilon/2$ is a constant of finite expansivity.
If $\card(\Gamma_{\epsilon/2}(y))=\infty$ for some $y\in X$,
then there is an accumulation point $x\in\Gamma_{\epsilon/2}(y)$. 
Since $\Gamma_{\epsilon/2}(y)\subset \Gamma_\epsilon(x)$ we have a contradiction because $x$ is not isolated in $\Gamma_\epsilon(x)$.
Then, each $\Gamma_{\epsilon/2}(y)$ is finite.

The second part is direct from the definitions.
\end{proof}

For the proofs of Theorems \ref{teoSepFin} and \ref{teosepRecFin} we introduce the following equivalence relation.
Given $\epsilon>0$ we say that $x,y\in X$ are $\epsilon$-\emph{related} 
 if there is a finite sequence $z_1,z_2,\dots,z_n\in X$ such that 
 $z_1=x$, $z_n=y$ and $\dist(z_k,z_{k+1})<\epsilon$ for all $k=1,2,\dots, n-1$.
 The class of a point $x$ will be denoted by $[x]_\epsilon$. 

 \begin{rmk}
 \label{rmkqcomptotdisc}
If $X$ is totally disconnected then for all $\expc>0$ 
 there is $\epsilon>0$ such that $\diam([x]_\epsilon)<\expc$ for all 
 $x\in X$.  
 \end{rmk}

\begin{teo}
\label{teoSepFin}
If $f\colon X\to X$ is a separating homeomorphism of a totally disconnected 
compact metric space
then $f$ is finite expansive and, moreover, 
there is $\epsilon>0$ such that if $\Gamma_\epsilon(x)\neq\{x\}$ then $x$ is periodic. 
\end{teo}

\begin{proof}
Let $\expc$ be a separating constant and take $\epsilon\in (0,\expc)$ from Remark \ref{rmkqcomptotdisc} such that $\diam([u]_\epsilon)<\expc$ for all $u\in X$. 
Arguing by contradiction, suppose that $f$ is not finite expansive. 
By Lemma \ref{lemFinExpGamma} there is $x\in X$ such that $\Gamma_\epsilon(x)$ accumulates in $x$. 
Consider the following subsets 
\[
 \left\{
 \begin{array}{ll}
A_0   & =\Gamma_\epsilon(x),\\ 
A_{n} & =\bigcup_{y\in A_{n-1}}\Gamma_\epsilon(y) \text{ if } n\geq 1.  
 \end{array}
 \right.
\]
Note that $\Gamma_\epsilon(y)\subset B_\epsilon(y)\subset [y]_\epsilon$ for all $y\in X$. 
In particular we have that $A_0\subset [x]_\epsilon$.
Suppose that $A_{n-1}\subset [x]_\epsilon$ and take $y\in A_{n-1}$. 
Since $\Gamma_\epsilon(y)\subset [y]_\epsilon$, and $[y]_\epsilon$ is a class of an equivalence relation, 
we conclude that $A_n\subset [x]_\epsilon$.
By induction we have that $A_n\subset [x]_\epsilon$ for all $n\geq 0$. 

Define $$A_*=\clos(\cup_{n\geq 0} A_n).$$
Given that each class $[x]_\epsilon$ is a closed subset, we conclude that $A_*\subset [x]_\epsilon$.
Note that $f^n(\Gamma_\epsilon(u))=\Gamma_\epsilon(f^n(u))$ for all $u\in X$ and all $n\in\Z$. Thus
$f^n(A_*)\subset [f^n(x)]_\epsilon$ and 
$\diam(f^n(A_*))\leq\expc$ for all $n\in\Z$.
As $\expc$ is a separating constant for $f$, for all $z\in A_n$ there is $k\in\Z$ such that $z=f^k(x)$. 
Then, $\Gamma_\epsilon(z)=f^k(\Gamma_\epsilon(x))$ accumulates in $z$, 
because $\Gamma_\epsilon(x)$ accumulates in $x$ and $f$ is a homeomorphism. 
Then, $A_*$ has no isolated points.
As the set $A_*$ is closed we conclude that $A_*$ is uncountable. 
Since $\diam(f^n(A_*))\leq\expc$ for all $n\in\Z$ we have that $\expc$ is not a constant of 
countable expansivity for $f$.
As $\expc$ is an arbitrary separating constant, we conclude that $f$ is not countably expansive. 
This contradicts that $f$ is separating (recall Table \ref{tablaExp}). Then, $f$ is finite expansive.

As before, consider $\epsilon$ from Remark \ref{rmkqcomptotdisc}.
 Take $x,y\in X$ such that 
 $x\neq y$ and $\dist(f^n(x),f^n(y))\leq \epsilon$ for all $n\in\Z$. 
 Since $f$ is separating we know that 
 $y=f^k(x)$ with $k\neq 0$. 
 Let $B=\{f^{nk}(x):n\in\Z\}$. 
 We know that $B\subset [x]_\epsilon$. 
 Moreover, $f^n(B)\subset [f^n(x)]_\epsilon$. 
 Consequently, $\diam(f^n(B))\leq\expc$ for all $n\in\Z$. 
 As we proved, $B$ is a finite subset. 
 This proves that $x$ is periodic.
\end{proof}

The following question remains open:
is Theorem \ref{teoSepFin} true if the space is not totally disconnected?
We obtain the following simplified table for totally disconnected spaces.
Note that every homeomorphism of a totally disconnected space is cw-expansive.

\begin{table}[h]
\[
\begin{array}{ccccc}
\text{exp} & \rightarrow & \text{separating} \\
\downarrow && \downarrow\\
N\text{-exp} & \rightarrow & \text{finite exp} 
& \rightarrow &  \text{countably exp}
\end{array}
\]
\caption{Hierarchy on totally disconnected spaces.}
\label{tablaExpTotDisc}
\end{table}

The next example (based on \cite{ArKinExp}*{Example 2.24}) gives a separating homeomorphism that is not $N$-expansive.
Examples of $N$-expansive homeomorphisms (of compact surfaces) not being separating can be found in 
\cite{ArRobNexp}.

\begin{exa}
\label{sepnoexp}
For each positive integer $n$ consider a subset $A_n\subset \R^+$ with $n$ elements 
such that $A_n\cap A_m=\emptyset$ if $m\neq n$ and $A_n\to\{0\}$ in the Hausdorff metric.
Suppose that $A_n=\{a_{n,i}:i\in \Z_n\}$, where $\Z_n=\Z/n\Z$ is the cyclic group with $n$ elements.
Let $X$ be the subset of the sphere $\R^2\cup\{\infty\}$ given by
$$
  X=\{\infty\}\cup(\Z\times\{0\})\cup
  \bigcup_{n\in\Z^+}([-n,n]\times A_n),
$$ 
where $[-n,n]$ denotes the interval of integers between $-n$ and $n$.
Define $f\colon X\to X$ as
$f(\infty)=\infty$, $f(n,0)=(n+1,0)$, 
$f(j,a_{n,i})=(j+1,a_{n,i})$ if $-n\leq j<n$ and 
$f(n,a_{n,i})=(-n,a_{n,i+1})$.
Recall that $i\in\Z_n$.
The homeomorphism $f$ is not $N$-expansive because 
for all $\epsilon>0$ there is $n\geq N$ such that
$\diam(f^k(\{0\}\times A_n))<\epsilon$, and $\{0\}\times A_n$ contains $n\geq N$ points.
It is a separating homeomorphism because these are the only points contradicting expansiveness and they are in the same (periodic) orbit.
\end{exa}

\begin{rmk}
\label{rmkPowers}
For a homeomorphism $f\colon X\to X$ of a compact metric space 
it holds that $f$ is expansive if and only if $f^n$ is expansive for all $n\neq 0$ (see \cite{Utz}).
It is easy to see that if $f^n$ is separating for some $n\neq 0$ then $f$ is separating. 
We remark that the converse is not true. 
The homeomorphism $f$ of Example \ref{sepnoexp} is separating but 
its powers are not. 
Given $k\geq 2$ we will show that $f^k$ is not separating. 
We continue with the notation of the example.
For a large integer $m$ consider $x=(0,a_{km,1})$ and $y=(0,a_{km,2})$. 
Note that $x,y$ have period $km(2km+1)$ by $f$ 
and $y=f^{2km+1}(x)$. 
As $k$ is not a divisor of $2mk+1$ we have that $x$ and $y$ are in different (periodic) orbits of $f^k$. 
Finally, given $\epsilon>0$ take $m$ sufficiently large so that $\dist(f^l(x),f^l(y))\leq\epsilon$ for all $l\in\Z$. Then, $f^k$ is not separating 
if $|k|\geq 2$.
\end{rmk}

Now we derive some consequences of Theorem \ref{teoSepFin}. 
We say that $f$ is \emph{minimal} if it contains no proper closed invariant subsets 
(equivalently, if every orbit is dense in $X$).
If $Y\subset X$ is a closed invariant subset and $f$ restricted to $Y$ is minimal we say that $Y$ is a minimal subset.

 \begin{cor}
 \label{corMin}
  Every minimal and separating homeomorphism of a compact metric space is expansive.
 \end{cor}

 \begin{proof}
As $f$ is separating, it is cw-expansive and we can apply \cite{Ka93}*{Theorem 5.2} to conclude that $X$ is totally disconnected. 
If $X$ is finite there is nothing to prove. If $\card(X)=\infty$ 
then there are not periodic orbits (because $f$ is minimal). 
Applying Theorem \ref{teoSepFin} we see that $f$ is expansive.
\end{proof}

We say that two points $x\neq y$ are \emph{positively asymptotic} 
if $\dist(f^n(x),f^n(y))\to 0$ as $n\to +\infty$. 

\begin{rmk}
\label{rmkTodosPer}
 If $f$ is a separating homeomorphism of a compact metric space and every point is periodic then 
 $X$ is a finite set. 
 This follows by the arguments in the proof of \cite{Utz}*{Theorem 2.4}.
\end{rmk}

For the following proof we recall that the $\omega$-\emph{limit} set of $x$ 
is the set $\omega(x)$ of points $y\in X$ for which there is $n_k\to+\infty$ such that 
$f^{n_k}(x)\to y$ as $k\to+\infty$.


\begin{cor}
\label{corSepAsymp}
 If $X$ is a compact metric space with $\card(X)=\infty$ 
 and $f$ is a separating homeomorphism of $X$ then 
 there are positively asymptotic points.
\end{cor}

\begin{proof}
By Remark \ref{rmkTodosPer} there is a point $x\in X$ that is not periodic. 
If $\omega(x)$ contains a periodic point $p$ then, as the positive orbit of $x$ acumulates on $p$, 
there is a point that is positively asymptotic to $p$.
Assume that $Y\subset \omega(x)$ is a minimal subset with infinitely many points.
By Corollary \ref{corMin} we know that $f$ is expansive on $Y$. 
By \cite{CK} we know that $f$ is not positively expansive in $Y$, and consequently, 
there are positively asymptotic points in $Y$. 
\end{proof}

A cw-expansive homeomorphism may not have positively asymptotic points if the space is totally disconnected. 
For example, the identity of a Cantor set is (trivially) cw-expansive but it has not asymptotic points. 
What can be said with respect to the existence of positively asymptotic points if $X$ is connected and $f$ is cw-expansive?
If $X$ is a non-trivial Peano continuum and $f$ is cw-expansive then there are asymptotic points, see \cite{Ka93}.
\section{Recurrence and expansive groups}

A homeomorphism $f\colon X\to X$ is 
\emph{recurrent} if for all $\epsilon>0$ there is 
$n\in\Z$, $n\neq 0$, such that $\dist(x,f^n(x))<\epsilon$ for all $x\in X$.

\begin{rmk}
In \cite{Br}*{Theorem 2}, Bryant proved that if $f$ is an expansive homeomorphism of a compact metric space 
and $f$ is recurrent then $X$ is a finite set. 
Bryant's proof is as follows. 
Suppose that for some $n\neq 0$ we have that 
$\dist(x,f^n(x))<\expc$ for all $x\in X$, where $\expc$ is an expansivity constant of $f$. 
Then, $$\dist(f^k(x),f^k(f^n(x)))=\dist(f^k(x),f^n(f^k(x)))<\expc$$ for all $k\in\Z$. 
The expansivity of $f$ implies that $x=f^n(x)$ for all $x\in X$. 
Then $X$ is finite. 
\end{rmk}

It is clear that this argument does not work if 
$f$ is finite expansive or separating instead of expansive. 
However, with different techniques, we will generalize Bryant's result for finite expansive and separating homeomorphisms.

\begin{thm}
\label{teosepRecFin}
Let $f$ be a recurrent homeomorphism of a compact metric space.
If $f$ is finite expansive or separating then $X$ is a finite set. 
\end{thm}

\begin{proof}
First we assume that $f$ is recurrent and finite expansive.
Denote by $\expc$ a finite expansivity constant.
 Let us argue by contradiction and assume that $X$ is not a finite set. 
As $f$ is finite expansive, we know that $f$ is cw-expansive.
 If $X$ contains a non-trivial connected subset, 
 by \cite{Ka93}*{Proposition 2.5} 
 there is a non-trivial continuum (a compact connected subset) $C\subset X$ such that $\diam(f^n(C))\to 0$ 
 as $n\to+\infty$ or $n\to-\infty$. This easily contradicts that $f$ is recurrent. 
 
 Now assume that $X$ is totally disconnected. 
 Suppose that $f$ is minimal. 
 Since $X$ is totally disconnected there is $\epsilon>0$ such that $\diam([x]_e)<\expc$ for all 
 $x\in X$.
 As $f$ is recurrent there is $m\in \Z$, $m\neq 0$, such that $\dist(x,f^m(x))<\epsilon$ for all $x\in X$.
 Then, $x$ and $f^{km}(x)$ are $\epsilon$-related for all $k\in\Z$ and $f^m\colon X\to X$ cannot be minimal. 
 Given that $f$ is minimal, $X$ can be decomposed as a disjoint union $X=\cup_{i=1}^mX_i$ such that 
 $f(X_i)=X_{i+1}$ (cyclically) and $f^m\colon X_i\to X_i$ is minimal. 
 Since $x$ and $f^{km}(x)$ are $\epsilon$-related for all $k\in\Z$ and for all $x\in X$, 
 we have that if $x\in X_i$ then $X_i\subset [x]_\epsilon$. 
 Then, $f$ cannot be separating because $\diam([x]_\epsilon)<\expc$ for all $x\in X$. 
 Therefore $f$ is not minimal. 
 
 The previous argument implies that every minimal subset of $X$ must be a periodic orbit. 
 Since $f$ is finite expansive we have that periodic orbits are dynamically isolated, that is, there is 
 $\epsilon>0$ such that if $\dist(x,y)<\epsilon$ and $x$ is a periodic point then there is 
 $n\in\Z$ such that $\dist(f^n(x),f^n(y))>\epsilon$. If a periodic point $x$ is an accumulation 
 point of $X$ then there is $y\neq x$ such that $\dist(f^k(x),f^k(y))\to 0$ as 
 $k\to+\infty$ or $k\to-\infty$. This contradicts that $f$ is recurrent. 
 Therefore, no periodic point is an accumulation point of $X$. 
 Since every $\omega$-limit set contains a minimal set, every point is periodic and $X$ has no accumulation points. 
 Then $X$ is a finite set.

Now assume that $f$ is separating and recurrent. 
If $\card(X)=\infty$ then we can apply Corollary \ref{corSepAsymp} 
to obtain two asymptotic points.
As recurrent homeomorphisms cannot have asymptotic points we arrive to a contradiction that proves that $X$ is finite.
 \end{proof}

Let $\homeos(X)$ denote the group of homeomorphisms of $X$. 
On $\homeos(X)$ consider the norm 
$$|f|=\sup_{x\in X}\dist(x,f(x)).$$ 
A subgroup $G\subset\homeos(X)$ is an \emph{expansive group} 
if for all $\epsilon>0$ there is $\delta>0$ such that 
if $\dist(f(x),f(y))<\expc$ for all $f\in G$ then 
there is $g\in G$ such that $y=g(x)$ and $|g|<\epsilon$.
This definition is related with the definition of 
kinematic expansive flow given in the introduction.

For $f\in\homeos(X)$ define the cyclic group 
$\gen f=\{f^n:n\in\Z\}$.
The key for the next result is to prove that if $f$ is expansive then $\gen f$ is discrete. 
Note that there are (non-expansive) cyclic groups that are not discrete, for example consider an irrational rotation of the circle.

\begin{cor}
\label{corCyc}
A homeomorphism $f$ of a compact metric space 
is expansive if and only if 
$\gen f\subset\homeos(X)$ is an expansive group. 
\end{cor}

\begin{proof}
The direct part follows by the definitions.
In order to prove the converse notice that 
the case of $X$ being finite is trivial. 
So, we will assume that $\card(X)=\infty$.
Note that if $\gen{f}$ is an expansive group then $f$ is separating. 
If $f$ were not 
expansive then $f$ should be recurrent. 
Applying Theorem \ref{teosepRecFin} we would arrive to a contradiction. 
This proves that $f$ is expansive.
\end{proof}

\noindent Departamento de Matemática y Estadística del Litoral, \\
Universidad de la República, Rivera 1350 Salto-Uruguay\\
E-mail: artigue@unorte.edu.uy
\end{document}